\documentclass[12pt,reqno]{amsart}
\usepackage{amsmath,amssymb,verbatim,amscd}

\theoremstyle{plain}
\newtheorem{theorem}{Theorem}[section]
\newtheorem{proposition}[theorem]{Proposition}

\theoremstyle{definition}

\newtheorem{corollary}[theorem]{Corollary}

\newtheorem{conjecture}[theorem]{Conjecture}

\newcommand{\ds}{\displaystyle}

\DeclareMathOperator{\Hom}{Hom}

\newcommand\mh{\mathcal{H}}

\newcommand\C{\mathbb{C}}

\newcommand\frh{\mathfrak{h}}
\newcommand\frg{\mathfrak{g}}
\newcommand\frb{\mathfrak{b}}

\begin{document}

\title{PRV for the fusion product, the case $\lambda\gg\mu$}
\begin{abstract}
Given a complex simple Lie algebra $\mathfrak g$ and a positive integer $\ell$,  under the assumption $\lambda\gg\mu$, we show that irreducible representations of $\mathfrak g$ of the form $V(\lambda+w\mu)$, $w\in W,$ with level at most $\ell$ appear in the fusion product of $V(\lambda)$ and $V(\mu)$.  This verifies the existence of PRV components for the fusion product in this special case.
\end{abstract}

\author{Arzu Boysal}
\address{Department of Mathematics\\ Bo\u{g}azi\c{c}i University\\
34342 \\ Istanbul\\ Turkey.}

\email{arzu.boysal@boun.edu.tr}

\date{}
\maketitle \pagestyle{myheadings}

\markboth{}{}  \maketitle

\section{Introduction}

Given a finite dimensional complex simple Lie algebra $\frg$, and two simple $\frg$-modules $V(\lambda)$ and $V(\mu)$, with dominant integral highest weights $\lambda$ and $\mu$ respectively,  there is a `largest' component in the decomposition of their tensor product $V(\lambda)\otimes V(\mu)$, due to Cartan, which is $V(\lambda +\mu)$ and occurs with multiplicity one.  There is also a `smallest' component, existence of which was proved  by Parthasarathy, Rao and Varadarajan (PRV) in 1960s \cite{PRV}; it is $V(\overline{\lambda+w_0\mu})$ and occurs again with multiplicity one.  Here $w_0$ denotes the longest element in the Weyl group $W$ of $\frg$, and $\overline{\nu}$ denotes the unique dominant element in the $W-$orbit of $\nu$.  

The (classical) PRV conjecture, seeking the existence of components in the tensor product decomposition of a similar form, where the longest element $w_0$ above is now replaced by an arbitrary Weyl group element, was first established by Kumar \cite{Ku1} (also proving a refined conjecture in \cite{Ku2}) and Mathieu \cite{Ma} independently.  More precisely, it is shown that, for any $w \in W$
the irreducible $\mathfrak{g}$-module
$V(\overline{\lambda+w \mu})$ occurs with multiplicity at
least one in $V({\lambda})\otimes V({\mu})$, where $\overline{\lambda+w \mu}$ denotes the unique dominant element in the $W-$orbit of $\lambda+w \mu$ (cf. Section \ref{s:tensorfusion}). Subsequently, weights of the form $\overline{\lambda+w \mu}$ are called PRV weights, and the corresponding representations $V(\overline{\lambda+w \mu})$ are called PRV components of $V({\lambda})\otimes V({\mu})$. (See Khare \cite{Kh} for a historical background of the problem).

For a finite dimensional complex simple Lie algebra $\frg$ with the additional data of a positive integer $\ell$,  one constructs the  corresponding (untwisted) affine Lie algebra $\tilde \frg$ with its central element acting via the scalar $\ell$.  Then, integrable irreducible representations of $\tilde \frg$ are parametrized by a finite set $P_\ell$.
There is an associated  product structure on two such simple $\tilde \frg$-modules of level $\ell$, called the \textit{fusion product}.  It is then very  natural to ask if the PRV components of level $\ell$ appear in the fusion product decomposition. (See Conjecture \ref{c:prvf} in Section \ref{s:tensorfusion} for the precise formulation of the PRV conjecture for the fusion product).  
For the affine Lie algebra corresponding to $\frg=\mathfrak{sl_n}$, this is true by works of Belkale on quantum Horn and saturation
conjectures \cite{Be}.  For the remaining simple Lie algebras the question is open in its full generality.  In literature, this is sometimes referred as the quantum counterpart of the PRV conjecture.

Here we give an elementary algebraic proof of the existence of such PRV components for any simple Lie algebra $\frg$  under the assumption $\lambda\gg\mu$, by which we mean that all weights of $V(\mu)$ shifted by $\lambda$ lie in the dominant chamber.  The proof of this special case should be seen as supporting evidence in the direction of the quantum analogue of the PRV conjecture for any simple Lie algebra $\frg$.   
Using the definition of the fusion product, formulated in terms of `truncated' tensor product of finite dimensional representations of $\frg$ that are of level at most $\ell$, our main result is as follows (Theorem \ref{c:prvfspecial} in the main text).

\begin{theorem}\label{thm:intro}
Suppose $\lambda, \mu \in P_{\ell}$, and  $\lambda\gg\mu$.  For any $w\in W$, if ${\lambda+w \mu}$ is in $P_{\ell}$, then $V({\lambda+w \mu})$ occurs with multiplicity one in the fusion product of $V({\lambda})$ and $V({\mu})$.
\end{theorem}

In the proof of this main result we make essential use of the algebra homomorphism from the representation ring of $\frg$ to the fusion ring of $\frg$ at level $\ell$ given by Faltings in \cite{F} for classical Lie algebras and $G_2$ (which is valid for any simple Lie algebra by works of Teleman \cite{Te} and Kumar \cite{Ku3}).  The distinguished element  $s_{\theta,l+\check{h}}$ of the affine Weyl group
associated to the highest root $\theta$ of $\frg$ plays a fundamental role in showing the nonvanishing result. 

A further consequence of Theorem \ref{thm:intro} is as follows.  For the same data $(\frg, \ell)$, and three dominant integral weights $\overrightarrow{\eta}=(\eta_1,\eta_2,\eta_3)$ in $P_\ell^3$
one constructs a vector bundle $\mathbb V_{\frg,\ell,\overrightarrow{\eta}}$ over the moduli space of $3$-pointed stable rational curves, whose fiber over a point is called the space of conformal blocks  \cite{TUY}.  Structure coefficients in the fusion product decomposition are intimately related to the dimensions of these spaces (cf.~Section \ref{section:fusion_defn}).  Denoting the dual of a weight $\nu$ by $\nu^*$, it follows from the theory of weights that if  $\nu$ is in $P_\ell$, then so is $\nu^*$.  With the above notation, our result demonstrates that, under the assumptions of Theorem \ref{thm:intro}, the  rank of the vector bundle $\mathbb V_{\frg,\ell,\overrightarrow{\eta}}$ with $\overrightarrow{\eta}=(\lambda,\mu, (\lambda+w \mu)^*)$ is equal to the dimension of the vector space of $\frg$-invariants of the tensor product of representations with highest weights corresponding to $\overrightarrow{\eta}$ (and both equal to one).  
For any triple  $\overrightarrow{\eta} \in P_\ell^3$ this equality of ranks necessarily holds for $\ell$ beyond a critical level. We remark that although our assumption in Theorem \ref{thm:intro} is restrictive on the distribution of weights, we do not impose such a bound on the level $\ell$. 

To describe our results precisely, we set up the notation and recall basic definitions and concepts in Section \ref{s:prelim}.  In Section \ref{s:tensorfusion} we recall the definition of the fusion product and its properties that are pertinent to the proof of the main result which is given in Section \ref{section:main}, where an explicit decomposition is also given at the end.

\section{Preliminaries}\label{s:prelim}
The main reference for this section is \cite{B}.

Let $G$ be a connected, simply connected, simple affine algebraic group over $\C$. We fix a Borel subgroup $B$ of $G$ and a maximal torus $T\subset B$. Let $\frh$,$\frb$ and $\frg$ denote the Lie algebra of
$T$, $B$ and $G$ respectively.  

Let $R=R(\frh,\frg)\subset \frh^*$ be
the root system; there is the root space decomposition $\frg=\frh
\oplus (\oplus_{\alpha \in R}\frg_{\alpha})$.  The choice of $\frb$ determines a set of simple roots $\Delta=\{\alpha_1, \dots, \alpha_r\}$ of $R$,
where $r$ is the rank of $G$.  Let $\frh_{\mathbb{R}}$ denote the real span of elements of $\frh$
dual to $\Delta$.  For each root $\alpha$, denote by $H_{\alpha}$ the
unique element of $[\frg_{\alpha},\frg_{-\alpha}]$ such that
$\alpha(H_{\alpha})=2$; it is called the coroot associated to the
root $\alpha$.
For any $\alpha \in R$, the Lie subalgebra $\frg_{\alpha}\oplus
[\frg_{\alpha},\frg_{-\alpha}] \oplus \frg_{-\alpha}$ is isomorphic
as a Lie algebra to $\mathfrak{sl}_2$ and thus will be denoted by
$\mathfrak{sl}_2(\alpha)$. 

Let $\{\omega_i \}_{1\leq i\leq r}$ be the set of fundamental
weights, defined as the basis of $\frh^*$ dual to
$\{H_{\alpha_i}\}_{1\leq i\leq r}$. We define the weight lattice
$P=\{ \lambda \in \frh^* : \lambda(H_{\alpha}) \in \mathbb{Z},\;
\forall \alpha \in R \}$, and denote the set of dominant weights by
$P^+$, that is,
$$P^+:=\{ \lambda \in \frh^* : \lambda(H_{\alpha_i}) \in
\mathbb{Z}_{\geq 0}\;
\forall \alpha_i \in \Delta \},$$ where $\mathbb{Z}_{\geq 0}$ is the set of
nonnegative integers.  The  set $P^+$ parametrizes the set of isomorphism
classes of all the finite dimensional irreducible representations of $\frg$.
For $\lambda\in P^+$, let $V(\lambda)$
be the associated  finite dimensional irreducible
representations of $\frg$ with highest weight $\lambda$. 

Let $\rho$ denote the sum of fundamental
weights and $\theta$ the highest root of
$\frg$.  Define $\check{h}(\frg):=1+\rho(H_{\theta})$, denoted in short by $\check{h}$; it is called the dual Coxeter number.


Let $(\;|\;)$ denote the Killing form on $\frg$ normalized such that
$(H_{\theta}|H_{\theta})=2$.  We will use the same notation for the restricted form on
$\frh$, and the induced form on $\frh^*$.  We introduce the Weyl group $W:=N_G(T)/T$.
The Killing form is positive definite on $\frh_{\mathbb{R}}$, it induces the following canonical isomorphism
\[F: \frh_{\mathbb{R}}^* \to \frh_{\mathbb{R}}\;;\;\; \alpha \mapsto (2/(H_{\alpha}| H_{\alpha}))H_{\alpha}\]
with $F^{-1}(H_\alpha)=(2/(\alpha|\alpha))\alpha$.
Under the above identification, $W$ is thought as generated by elements $w_{\alpha}$ for $\alpha \in \Delta$; acting
on $\frh_{\mathbb{R}}^*$ as $$w_{\alpha}\beta=\beta-\beta(H_{\alpha})\alpha.$$

Denote $P_{\mathbb R}:=P\otimes_{\mathbb Z}\mathbb R$. Any root
$\alpha$ in $R$  defines a {\it wall} $H_{\alpha}=\{\lambda\in P_{\mathbb R}: (\lambda|\alpha)=0\}.$
The closures of the connected components of $P_{\mathbb R}\setminus
\bigl(\cup_{\alpha\in R}\,H_{\alpha}\bigr)$ are called {\it chambers}.
Then, any chamber is a fundamental domain for the action of $W$ on  $P_{\mathbb R}$.  The chamber $C$ defined by the conditions $\{(\lambda|\alpha_i)\geq 0, i \in\Delta\}$ is called the \textit{dominant chamber}.

Let $w_0$ denote the longest element in the Weyl group.  Then for any $\lambda \in P$, the dual of $\lambda$
denoted by $\lambda^*$ is $-w_0\lambda$.   For $\lambda\in P^+, \lambda^*$ is again in $P^+$ and, moreover,
 $V(\lambda)^*\simeq V(\lambda^*).$


Let $\tilde{\frg}=\frg \otimes \mathbb{C}((z))\oplus
\mathbb{C}\mbox{C}$ denote the (untwisted) affine Lie algebra
associated to $\frg$ over $\mathbb{C}((z))$ (where $\mathbb{C}((z))$
denotes the field of Laurent polynomials in one variable $z$), with
the Lie bracket
\[[x\otimes f,y\otimes g]=[x,y]\otimes fg+ (x|y) Res_{z=0}(gdf)\cdot \mbox{C}
\;\mbox{and}\;[\frg,\mbox{C}]=0\] for $x,y \in \frg$ and $f,g \in
\mathbb{C}((z))$.


We fix a positive integer $\ell$, called the {\it level}, and define
the set \[P_{\ell}(G):=\{\lambda \in P^+ : \lambda(H_{\theta})\leq
\ell\};\] it is finite and closed under taking the dual; we denote it in short by $P_{\ell}$. For any
$\lambda \in P^+$, the integer $\lambda(H_{\theta})$ is said to be
the {\it level of the representation} $V({\lambda})$.

It is well known from representation theory of affine Lie algebras
that for each $\lambda \in P_{\ell}$ there exists a unique (upto
isomorphism) left $\tilde{\frg}$-module $\mh({\lambda})$, called
the integrable highest weight $\tilde{\frg}$-module, with the
central element $\mbox{C}$ acting as $\ell \cdot \mbox{Id}$ (see
e.~g.~\cite{K}).

\section{Tensor product, fusion product and PRV components}\label{s:tensorfusion}

Let $\mathcal{R}(\frg)$ denote the ring of finite dimensional representations of $\frg$; it is a free
$\mathbb{Z}$-module with basis $\{V({\lambda}):\;\lambda \in P^+\}$ and product structure
\[V(\lambda)\otimes V(\mu)=\ds \sum_{\nu \in P^+} m_{\lambda
\mu}^{\nu} V(\nu),\] where $m_{\lambda \mu}^{\nu}=\dim
\Hom_{\frg}(V(\lambda)\otimes V(\mu)\otimes V(\nu^*),
\mathbb{C})$, with $\mathbb{C}$ thought as a trivial $\frg$
module.

As the dominant chamber $C$ is a fundamental domain for the action of $W$ on $P_{\mathbb R}$,  any $\xi\in P$ can be expressed as $w\overline{\xi}$ with $w\in W$ and $\overline{\xi}\in P^+=P\cap C$.  The dominant weight $\overline{\xi}$ in the $W-$ orbit of $\xi$ is uniquely determined.

With the above notation, the classical Parthasarathy--Ranga Rao--Varadarajan (PRV) conjecture (already proven, see \cite{Ku1,Ku2} or \cite{Ma}) is as follows:
\begin{theorem} Let
$V({\lambda_1})$ and $V({\lambda_2})$ be two finite dimensional
irreducible $\mathfrak{g}$-modules with highest weights
$\lambda_1$ and $\lambda_2$ respectively. Then, for any $w \in W$
the irreducible $\mathfrak{g}$-module
$V(\overline{\lambda_1+w \lambda_2})$ occurs with multiplicity at
least one in $V({\lambda_1})\otimes V({\lambda_2})$, where
$\overline{\lambda_1+w \lambda_2}$ denotes the unique dominant
element in the $W-$orbit of $\lambda_1+w \lambda_2$.  
\end{theorem}
Weights of the form $\overline{\lambda_1+w \lambda_2}$ are called PRV weights, and corresponding components $V(\overline{\lambda_1+w \lambda_2})$ are called PRV components of the tensor product $V({\lambda_1})\otimes V({\lambda_2})$.

We are interested in a natural generalization of the above theorem for the decomposition of the {\it fusion product} of two  
integrable irreducible representations of $\tilde{\frg}$ of level  $\ell$.  There are various equivalent definitions of the fusion product (see for example \cite{BK}), 
we will use the one that allows us to formulate the problem in terms of a `truncated product' of finite dimensional representations of $\frg$ 
that have level at most $\ell$.  

\subsection{Definition of fusion product}\label{section:fusion_defn}
The fusion ring associated to $\mathfrak{g}$ and a fixed nonnegative integer $\ell$,
$\mathcal{R}_{\ell}(\mathfrak{g})$, is a free $\mathbb{Z}$-module
with basis $\{ V(\lambda): \lambda \in P_{\ell} \}$. The ring structure, called the {\it fusion product}, is
defined as follows:
$$V(\lambda)\otimes ^F V(\mu):= \ds \bigoplus_{\nu \in P_{\ell}^{+}}
n_{\lambda,\mu}^{\nu} V(\nu),$$ where $n_{\lambda,\mu}^{\nu}$ is
the dimension of 
$$V_{\mathbb{P}^1}^\dag(\lambda,\nu,\mu^*):=\Hom_{\frg\otimes\mathcal{O}(\mathbb{P}^1-\overrightarrow{p})}(\mh (\lambda)\otimes
\mh (\mu)\otimes\mh (\nu^*),\mathbb{C}).$$ Above
$\frg\otimes\mathcal{O}(\mathbb{P}^1-\overrightarrow{p})$ denotes the Lie algebra of $\frg$-valued regular functions on the Riemann
sphere punctured at 3 distinct points $\overrightarrow{p}=\{p_1,p_2,p_3\}$.  We refer to \cite{TUY} for the precise action of $\frg\otimes\mathcal{O}(\mathbb{P}^1-\overrightarrow{p})$
on $\mh (\lambda)\otimes \mh (\mu)\otimes\mh (\nu^*)$, above $\mathbb{C}$ is thought as a trivial module for the algebra.  The space $V_{\mathbb{P}^1}^\dag(\lambda,\nu,\mu^*)$ is called the {\it space of conformal blocks} on $\mathbb{P}^1$ with three marked points $\overrightarrow{p}$ and weights
$\lambda,\mu,\nu^*$ attached to them with central charge $\ell$; its dimension, that is $n_{\lambda,\mu}^{\nu}$, is given by the Verlinde formula.

By its definition the fusion product $\otimes^F$ is commutative; it is also associative as a result of the `factorization rules' of \cite{TUY}.
Furthermore, the canonical map
\[V_{\mathbb{P}^1}^\dag(\lambda,\nu,\mu^*)\to \Hom_{\frg}(V(\lambda)\otimes V(\mu)\otimes
V(\nu^*), \mathbb{C})\] induced from the natural inclusion
$V(\lambda)\otimes V(\mu)\otimes V(\nu^*)\hookrightarrow
\mh (\lambda)\otimes \mh (\mu)\otimes\mh (\nu^*)$ is an injection
\cite{SU}.  In particular, for any $\ell \in \mathbb{Z}_{>0}$, the
inequality 
\begin{equation}\label{ineq:tensorfusion}
n_{\lambda,\mu}^{\nu}\leq m_{\lambda,\mu}^{\nu}
\end{equation}
holds.

We introduce some additional notation.

Let $W_\ell$ be the group of
affine transformations of $P_{\mathbb R}$ generated by the (finite) Weyl group $W$ and the translation
$\lambda\mapsto \lambda+(\ell+\check{h})\theta$. Then, $W_\ell$ is the semi-direct product
of $W$ by the lattice $(\ell+\check{h})Q^{\text{long}}$, where
$Q^{\text{long}}$ is the sublattice of $P$ generated by the long roots. 
For any root
$\alpha\in R$ and $n\in \mathbb Z$, define the {\it affine wall}
\[H_{\alpha,n}=\{\lambda\in P_{\mathbb R}: (\lambda|\alpha)=n(\ell+\check{h})\}.\]
The closures of the connected components of $P_{\mathbb R}\setminus
\bigl(\cup_{\alpha\in R, n\in \mathbb Z}\,H_{\alpha,n}\bigr)$ are called the {\it alcoves}.
Then, any alcove is a fundamental domain for the action of $W_\ell$. The
{\it fundamental alcove} is by definition
\[A_\ell^o=\{\lambda\in P_{\mathbb{R}}: \lambda(H_{\alpha_i})\geq 0\, \forall\,
\alpha_i \in \Delta, \,\text{and}\, \lambda(H_\theta)\leq \ell+\check{h}\}.\]
We remark that weights in the interior of the alcove 
$A_\ell^o$ are  of the form $\lambda +\rho$ for $\lambda \in P_\ell$.


\subsection{Definition of $\pi :  \mathcal{R}(\frg) \to   \mathcal{R}_{\ell}(\frg)$}\label{section:defn_pi} 
Following Beauville \cite{Bea}, we define the $\mathbb{Z}$-linear map $\pi$ as follows:  
$$\pi(V({\lambda}))= \begin{cases} 
0, & \text{if}\; \lambda+\rho\;\;\text{lies on an affine wall}\\  
\varepsilon(w)V(\mu) & \text{otherwise}
\end{cases}
$$
 where
$\mu$ is the unique element of $P_\ell$ and $w$ of $W_\ell$ such that  $\lambda+\rho=w(\mu+\rho)$. Here  $\varepsilon(w)$ is the sign of
the affine Weyl group element $w$.  With the above definition, we recall the following theorem which is due to Faltings  \cite{F} for classical Lie algebras and $G_2$. It is valid for any simple Lie algebra by works of Teleman \cite{Te} and Kumar \cite{Ku3}.

\begin{theorem} \label{t:fhom}
The $\mathbb{Z}$-linear map
 $\pi :  \mathcal{R}(\frg) \to   \mathcal{R}_{\ell}(\frg)$  is an algebra homomorphism with respect to
the fusion product on $ R_\ell(\frg)$.
\end{theorem}
We will use this theorem in a fundamental way in the next section.

Now, as $A_\ell^o$ is a fundamental domain for the action of $W_\ell$ on $P_{\mathbb R}$,  given $\lambda_1,\lambda_2 \in P_\ell$ and $w\in W_\ell$, there exits $\tilde w \in W_\ell$ such that
$$\tilde w\big(\frac{\ell+\check{h}}{\ell}\lambda_1+w(\frac{\ell+\check{h}}{\ell}\lambda_2)\Big) \in A_\ell^o$$ is the unique element in the $W_\ell$ orbit of $\frac{\ell+\check{h}}{\ell}\lambda_1+w(\frac{\ell+\check{h}}{\ell}\lambda_2)$.  Then, by construction, $\frac{\ell}{\ell+\check{h}}(\tilde w(\frac{\ell+\check{h}}{\ell}\lambda_1+w(\frac{\ell+\check{h}}{\ell}\lambda_2)))$ is in $P_\ell$, and it can be expressed as
$\tilde w'(\lambda_1+w'(\lambda_2))\;(\text{mod}\;\ell Q^{\text{long}})$ for some $\tilde w', w'\in W$.  We denote this element in $P_\ell$ (uniquely determined by the triple $\lambda_1,\lambda_2$ and $w$) by $\overline{\lambda_1+w \lambda_2}^F$.

In the special case that $\lambda_1,\lambda_2\in P_\ell$ and $w\in W$ satisfying $\lambda_1+w\lambda_2\in P_\ell$, we have $\overline{\lambda_1+w \lambda_2}^F=\lambda_1+w \lambda_2$.

With the notation as above, the PRV conjecture for the fusion product is as follows: 
\begin{conjecture}\label{c:prvf}
Suppose $\lambda$ and $\mu$ are in $P_{\ell}$. Then, for any $w \in W_\ell$
the irreducible $\mathfrak{g}$-module $V(\overline{\lambda+w \mu}^F)$ occurs with multiplicity at
least one in $V({\lambda})\otimes^F V({\mu})$.  
\end{conjecture}

For $\frg=\mathfrak{sl_n}$, the conjecture is true by works of Belkale on quantum Horn and saturation
conjectures \cite{Be}.

\section{The case $\lambda\gg\mu$}\label{section:main}
Suppose $\lambda$ and $\mu$ are in $P_\ell$.  Consider the weight space decomposition $$V(\mu)=\displaystyle \sum_{\nu\in \frh^*}m_\mu(\nu)V_\nu$$ where $V_\nu=\{v \in V(\mu): h\cdot v=\nu (h)v\;\text{for all}\; h\in \frh\}$ and $m_\mu(\nu)$ is the dimension of the weight space $V_\nu$.
Let $\Pi(\mu)$ denote the set of all weights of $V(\mu)$.  We are looking at the special case $\lambda\gg\mu$, by that we mean $\lambda +\nu\in P^+$ for all $\nu \in  \Pi(\mu)$.
Thus, under the assumption $\lambda\gg\mu$,  PRV components of the tensor product $V(\lambda)\otimes V(\mu)$  are precisely  $\{V(\overline{\lambda+w \mu})=V(\lambda+w \mu), \;w\in W\}$. 

We will prove the following formulation of the problem in this special case.  

\begin{theorem}\label{c:prvfspecial}
Suppose $\lambda, \mu \in P_{\ell}$, and  $\lambda\gg\mu$.  

For any $w\in W$, if ${\lambda+w \mu}$ is in $P_{\ell}$, then $V({\lambda+w \mu})$ occurs with multiplicity one in $V({\lambda})\otimes^F V({\mu})$.
\end{theorem}

The proof the above theorem will be given at the end of this section.  
To establish nonvanishing of structure coefficients for PRV components, and to determine the precise multiplicity we investigate the tensor product structure in further detail.

Let $\mathbb{Z}[P]$ be the group ring of $P$ with canonical basis $\{e^\mu, \mu \in P\}$.
Consider the character homomorphism $\chi: \mathcal{R}(\frg) \to
\mathbb{Z}[P]$ mapping $V \mapsto \sum_{\mu\in P}(\dim(V_{\mu}))e^{\mu}$.
For $\lambda$ dominant integral, denote $\chi(V(\lambda))$ in short by $\chi_\lambda$; it is given by
the Weyl character formula
$$\chi_{\lambda}=\frac{\sum_{w \in W}\varepsilon(w) e^{w(\lambda+\rho)}}
{\sum_{w \in W} \varepsilon(w) e^{w\rho}}.$$
Denote $D:=\sum_{w \in W} \varepsilon(w) e^{w\rho}$.  Then, using the fact that  the elements of $\Pi(\mu)$ are permuted by $W$ and $\text{dim} V_\nu=\text{dim}V_{w \nu}$ for any $w\in W$, 
\[
\chi_\lambda \chi_\mu =\displaystyle \sum_{\nu \in \Pi(\mu)} m_\mu(\nu) e^\nu  D^{-1}\sum_{w \in W} \varepsilon(w) e^{w(\lambda+\rho)}= D^{-1} \displaystyle \sum_{w \in W}\sum_{\nu \in \Pi(\mu)}m_\mu(\nu) \varepsilon(w) e^{w(\lambda+\nu+\rho)}
\]
Now interchanging the order of summation, as $\lambda +\nu$ is dominant integral for all $\nu \in  \Pi(\mu)$, we may use the Weyl character formula again
and conclude that
$$\chi_\lambda \chi_\mu= \displaystyle \sum_{\nu \in \Pi(\mu)} m_\mu(\nu) D^{-1}\sum_{w \in W} \varepsilon(w) e^{w(\lambda+\nu+\rho)}
=\displaystyle \sum_{\nu \in \Pi(\mu)} m_\mu(\nu) \chi_{\lambda+\nu}.$$
The above expression is a special case of a formula attributed to Brauer and Kliymk independently. (See excercise 9 in Section 24 of Humphreys 
\cite{H} for the formulation of the general case, and the historical references therein.) 
Thus the tensor product decomposition of $V(\lambda)$ and $V(\mu)$ is very explicit when $\lambda\gg\mu$, it is
\begin{equation}\label{eq:tensor}
V(\lambda)\otimes V(\mu)=\displaystyle \sum_{\nu \in \Pi(\mu)}m_\mu(\nu) V(\lambda+\nu).
\end{equation}

As any component of $V(\lambda)\otimes V(\mu)$ has level at most $2\ell$, any highest weight $\lambda+\nu$ of the decomposition in (\ref{eq:tensor}) that does not lie in $P_\ell$ satisfies the inequality $\ell+1 \leq (\lambda+\nu)(H_{\theta})\leq 2\ell$.  

We start with grouping the terms on the right hand side of  (\ref{eq:tensor}) as follows.  We call
$$S:=\{\nu \in \Pi(\mu): \nu \text{ is not of the form } w\mu \text{ for any } w\in W\}.$$
Then, using the fact that $\text{dim} V_\mu=\text{dim}V_{w \mu}=1, \; (w\in W)$ in $V(\mu)$,
\begin{equation}\label{eq:tensor2}
V(\lambda)\otimes V(\mu)=\displaystyle \sum_{\nu \in S} m_\mu(\nu) V(\lambda+\nu)+\sum_{w \in W}  V(\lambda+w\mu).\nonumber
\end{equation}
We further group the representations above in terms of their levels as
\begin{multline}\label{eq:tensorgroup}
V(\lambda)\otimes V(\mu) 
=\displaystyle \sum_{\substack{\{\nu \in S:\\ \lambda+\nu \in P_\ell\}}} m_\mu(\nu) V(\lambda+\nu)+
\sum_{\substack{\{\nu \in S:\\ (\lambda+\nu)(H_\theta)=\ell+1\}}} m_\mu(\nu) V(\lambda+\nu)+\\
\displaystyle \sum_{\substack{\{\nu \in S: \\ \ell+1<(\lambda+\nu)(H_\theta)\leq 2\ell\}}} m_\mu(\nu)V(\lambda+\nu)+
\sum_{\substack{\{w \in W: \\ \lambda+w\mu \in P_\ell\}}}  V(\lambda+w\mu) +\\
\displaystyle \sum_{\substack{\{w \in W: \\ (\lambda+w\mu)(H_\theta)=\ell+1\}}} m_\mu(\nu)V(\lambda+w\mu)+\sum_{\substack{\{w \in W: \\ \ell+1 <(\lambda+w\mu)(H_\theta)\leq 2\ell\}}}  V(\lambda+w\mu).
\end{multline}

We now investigate the image of the tensor product decomposition as expressed in (\ref{eq:tensorgroup}) under the homomorphism $\pi$ of Theorem \ref{t:fhom}.

 Any representation $V(\xi)$ in (\ref{eq:tensorgroup}) with highest weight $\xi$ satisfying $(\xi+\rho|\theta)=\ell+\check{h}$ (i.e.~any representation $V(\xi)$ of level $\xi(H_\theta)=\ell+1$) is mapped to zero under $\pi$.  If $\xi$ is already in $P_\ell$, then $\pi(V(\xi))=V(\xi)$.  We now further investigate the remaining components  of the above decomposition, that is,  representations $V(\xi)$ in (\ref{eq:tensorgroup}) whose level is strickly greater than $\ell+1$.  

Consider the following element of $W_\ell$
$$s_{\theta,\ell+\check{h}}(\xi):=s_\theta(\xi)+(\ell+\check{h})\theta.$$ 
Its shifted action is given by
$$s_{\theta,\ell+\check{h}}\cdot \xi:=s_{\theta,\ell+\check{h}}(\xi+\rho)-\rho=s_\theta(\xi)+(\ell+1)\theta.$$
The following proposition reveals the significance of this distinguished element under the assumption $\lambda\gg\mu$.

\begin{proposition}\label{pro:shiftedwlevel}
Suppose $\lambda, \mu \in P_\ell$ and $\lambda\gg\mu$.

For any $V(\xi)$ in the decomposition of $V(\lambda)\otimes V(\mu)$ with 
\begin{equation}\label{in:level}
\ell+1 <\xi(H_\theta)\leq 2\ell,
\end{equation}
we have that $s_{\theta,l+\check{h}}\cdot\xi$ is in $P_\ell$. 

Moreover, $V(s_{\theta,l+\check{h}}\cdot\xi)$ is not of the form $V(\lambda+w\mu)$ for any $w\in W$, that is, it is not a PRV component for the tensor product.
\end{proposition}

\begin{proof}
The level of $s_{\theta,\ell+\check{h}}\cdot\xi$ is given by
$$(s_{\theta,\ell+\check{h}}\cdot\xi)(H_\theta)=(\xi-\xi(H_\theta)\theta+(\ell+1)\theta)(H_\theta)=-\xi(H_\theta)+2(\ell+1).$$
Using the above equation together with inequality (\ref{in:level}), we get that 
$$2\leq (s_{\theta,\ell+\check{h}}\cdot\xi)(H_\theta)\leq \ell.$$
Thus to prove that $s_{\theta,l+\check{h}}\cdot\xi$ is in $P_\ell$, it suffices to show that it is dominant.
We will show this in two cases; though this separation of cases is not necessary, it makes the exposition clearer.

{\it Case 1:} Suppose $\xi$ is of the form $\xi=\lambda+w\mu$ ($w\in W$), satisfying inequality (\ref{in:level}). 
We then necessarily have  $(w\mu)(H_\theta)>0$ (as $\lambda \in P_\ell$), and thus $w\mu-s_\theta(w\mu)$ will be a positive integer multiple of $\theta$.

Now consider weights in $\Pi(\mu)$ of the form $w\mu+k\theta$ $(k\in \mathbb Z)$. 
The subspace of $V(\mu)$ spanned by the weight spaces $V_{w\mu+k\theta}\; (k\in \mathbb Z)$ is invariant under $\mathfrak{sl}_2(\theta)$, thus the 
weights in $\Pi(\mu)$ of the form
\begin{equation}\label{string1}
w\mu, w\mu-\theta, \cdots, w\mu-(w\mu)(H_\theta)\theta=s_\theta(w\mu)
\end{equation}
form a connected (i.e.~unbroken) string.
We will show that
$$s_{\theta,\ell+\check{h}}\cdot(\lambda+w\mu)=\lambda +w\mu-(\lambda +w\mu)(H_\theta)\theta+(\ell+1)\theta$$ 
is in the connected string 
\begin{equation}\label{shiftedstring}
\lambda+w\mu, \lambda+w\mu-\theta, \cdots, \lambda+w\mu+k\theta, \cdots,\lambda+s_\theta(w\mu),
\end{equation}
which is obtained from (\ref{string1}) by adding $\lambda$ to each element.
This will verify that $s_{\theta,\ell+\check{h}}\cdot(\lambda+w\mu)$ is in the dominant chamber, as all weights in the string (\ref{shiftedstring}) are dominant by our assumption $\lambda\gg\mu$.

Now express $s_{\theta,\ell+\check{h}}\cdot(\lambda+w\mu)=\lambda+w\mu+m\theta,$ where
$$m= -(\lambda +w\mu)(H_\theta)+(\ell+1).$$  To show that $s_{\theta,\ell+\check{h}}\cdot(\lambda+w\mu)$ is in the connected string (\ref{shiftedstring}), it suffices to show that 
\begin{equation}\label{in:case1}
-(w\mu)(H_\theta)\leq m\leq 0.
\end{equation}
As $\lambda+w\mu$ satisfies inequality (\ref{in:level}) we immediately get that $m<0$. Also, as $\lambda \in P_\ell$, we have that $m=(\ell+1-\lambda(H_\theta))-(w\mu)(H_\theta)> -(w\mu)(H_\theta)$. Thus both inequalities in (\ref{in:case1}) are satisfied, and in fact strictly.

{\it Case 2:} Suppose $\xi$ is of the form $\xi=\lambda+\nu$ with $\nu\in S$ satisfying inequality (\ref{in:level}). As $\lambda \in P_\ell$, we again necessarily have  $\nu(H_\theta)>0$. 

As in case 1,  the weights in $\Pi(\mu)$ of the form $\nu+k\theta$ form the connected string (weights of any simple $\mathfrak{sl}_2(\theta)$ submodule of $V(\mu)$ containing $\nu$)
\begin{equation}
\nu+q\theta, \cdots, \nu, \cdots, \nu-r\theta,\nonumber
\end{equation}
where $r, q$ are nonnegative integers satisfying $r-q=\nu(H_\theta)>0$. 
We will show that
$$s_{\theta,\ell+\check{h}}\cdot(\lambda+\nu)=\lambda +\nu-(\lambda +\nu)(H_\theta)\theta+(\ell+1)\theta$$ 
is in the $\lambda$-shifted connected string
\begin{equation}\label{shiftedstring2}
\lambda+\nu+q\theta, \cdots, \lambda+\nu, \cdots, \lambda+\nu-r\theta.
\end{equation}
This will verify that $s_{\theta,\ell+\check{h}}\cdot(\lambda+\nu)$ is in the dominant chamber, as all weights in the string (\ref{shiftedstring2}) are dominant, again by the assumption $\lambda \gg\mu$.

Now express $s_{\theta,\ell+\check{h}}\cdot(\lambda+\nu)=\lambda +\nu+n\theta$ with $n=\ell+1-(\lambda+\nu)(H_\theta)$. To show that $s_{\theta,\ell+\check{h}}\cdot(\lambda+\nu)$ is in the connected string (\ref{shiftedstring2}), it suffices to show that 
\begin{equation}\label{in:case2}
-r\leq n\leq q.
\end{equation}
As $\lambda+\nu$ satisfies inequality (\ref{in:level}), we have $n<0$ and hence $n<q$ as $q\geq0$. Expressing $r=q+\nu(H_\theta)$, and using the fact that $\lambda \in P_\ell$, we immediately get  $-r<n$.  Thus both inequalities in (\ref{in:case2}) are (strictly) satisfied.

To complete the proof of the proposition what remains to show is that $s_{\theta,l+\check{h}}\cdot\xi$ is not of the form $\lambda+w\mu$ for any $w$ in $W$, and thus is not a PRV weight for the tensor product.  We will now explain how this immediately follows from our calculations above revealing that inequalities (\ref{in:case1}) and (\ref{in:case2}) are strictly satisfied.

Using the explicit form of the tensor product decomposition (\ref{eq:tensor}), highest weights of  PRV components appearing in $V(\lambda)\otimes V(\mu)$ are in one to one correspondence with vertices of the convex hull of the set of weights of $V(\mu)$ translated by $\lambda$. 
For any such vertex $v$, simply because it lies on the boundary of the $\lambda$-translated convex hull of $\Pi(\mu)$, $v-\lambda \in \Pi(\mu)$ is necessarily a weight at one of the end points of any root strings in $\Pi(\mu)$ containing it. Then, 
to prove the claim it suffices to show that $s_{\theta,l+\check{h}}\cdot\xi$ does not lie at any of the end points of the $\theta$-string (as in equations (\ref{shiftedstring}) or (\ref{shiftedstring2})) that it is  contained.

To show this, let us first suppose $\eta=s_{\theta,\ell+\check{h}}\cdot\xi$ and  $\xi=\lambda+w\mu$  ($w\in W$) as in Case 1.  Then, as $\eta=\lambda+w\mu+m\theta$ lies on the string (\ref{shiftedstring}), 
the only possible vertex of the $\lambda$-shifted convex hull of $\Pi(\mu)$ that $\eta$ can be is $\lambda+ s_\theta(w\mu)=\lambda+w\mu -(w\mu)(H_\theta)\theta$.  But in Case 1 above, we showed that $m> -(w\mu)(H_\theta)$  which implies that $\eta\neq \lambda+w\mu -(w\mu)(H_\theta)\theta$, hence $\eta$ is not a PRV weight. (As $m<0$, $\eta$ is clearly  not equal to $\lambda+w\mu$.)

Similarly, suppose $\eta=s_{\theta,\ell+\check{h}}\cdot\xi$ and $\xi=\lambda+\nu$  with $\nu \in S$ as in Case 2.  Consider again the $\theta$-string (\ref{shiftedstring2}) that $\eta=\lambda+\nu+n\theta$ is an element of.
The only possible vertices of the $\lambda$-shifted convex hull of the weights of $V(\mu)$ that may be on that string would be the initial weight $\lambda+\nu-r\theta$ or the terminal weight $\lambda+\nu+q\theta$ (possibly neither).  In Case 2 above we showed that $-r<n<q$, which implies that $\eta$ is neither of those weights, thus  it is not a vertex, and hence cannot be a PRV weight.
\end{proof}


The following result immediately follows from the definition of $\pi$.
\begin{corollary}\label{co:thetashift}
Suppose $\lambda, \mu \in P_\ell$ and $\lambda\gg\mu$. For any $V(\xi)$ in the decomposition of $V(\lambda)\otimes V(\mu)$ with $\ell+1 <\xi(H_\theta)\leq 2\ell$, we have $\pi(V(\xi))=-V(s_{\theta,l+\check{h}}\cdot\xi)$.
\end{corollary}

By Corollary \ref{co:thetashift} and Theorem \ref{t:fhom}, we have the following very explicit expression for the fusion product of $V(\lambda)$ and $V(\mu)$ (as virtual modules) whenever $\lambda\gg\mu$ 
\begin{multline}\label{eq:fusionp}
V(\lambda)\otimes^F V(\mu) 
=\displaystyle \sum_{\substack{\{\nu \in S:\\ \lambda+\nu \in P_\ell\}}} m_\mu(\nu) V(\lambda+\nu)-
\sum_{\substack{\{\nu \in S: \\ \ell+1<(\lambda+\nu)(H_\theta)\leq 2\ell\}}} m_\mu(\nu)V(s_{\theta,\ell+\check{h}}\cdot(\lambda+\nu))\\
+\displaystyle \sum_{\substack{\{w \in W: \\ \lambda+w\mu \in P_\ell\}}}  V(\lambda+w\mu)-\sum_{\substack{\{w \in W: \\ \ell+1 <(\lambda+w\mu)(H_\theta)\leq 2\ell\}}}  V(s_{\theta,\ell+\check{h}}\cdot(\lambda+w\mu)).
\end{multline}

\paragraph{Proof of Theorem \ref{c:prvfspecial}.}
By Proposition \ref{pro:shiftedwlevel},  none of the weights in decomposition (\ref{eq:fusionp}) belonging to the set
\begin{multline}
\tilde S:=\{s_{\theta,\ell+\check{h}}\cdot(\lambda+w\mu):\ell+1 <(\lambda+w\mu)(H_\theta)\leq 2\ell,\; w\in W \}\cup\\ \nonumber
\{s_{\theta,\ell+\check{h}}\cdot(\lambda+\nu):\ell+1<(\lambda+\nu)(H_\theta)\leq 2\ell,\; \nu \in S \}
\end{multline}
\noindent is of the form $\lambda+w\mu$ for any $w\in W$.  Then, using equation (\ref{eq:fusionp}) we get that the representations $\{V(\lambda+w\mu): \lambda+w\mu \in P_\ell\}$ are preserved in the fusion product decomposition, that is, no `cancellation' by virtual modules $-V(\eta)$, $\eta \in \tilde S$.  This proves the theorem.

More explicitly, for $\lambda\gg\mu$, by Theorem \ref{t:fhom}, inequality (\ref{ineq:tensorfusion}), and Theorem \ref{c:prvfspecial},  we have
\begin{multline}\label{eq:fusionson}
V(\lambda)\otimes^F V(\mu) 
=\displaystyle \sum_{\substack{\nu \in S\\ \lambda+\nu \in P_\ell}} \tilde{m}_\mu(\nu) V(\lambda+\nu)
+\displaystyle \sum_{\substack{w \in W \\ \lambda+w\mu \in P_\ell}}  V(\lambda+w\mu) .
\end{multline}
for some nonnegative integer $\tilde{m}_\mu(\nu)=m_\mu(\nu)-m_\mu(\beta)$, where $\beta \in \Pi(\mu)$ satisfies
$$s_{\theta,\ell+\check{h}}^{-1}(\lambda+\beta+\rho)-\lambda-\rho=\nu\;\;\text{and}\;\; \lambda+\beta \in P_\ell.$$
(If no such $\beta$ exists we take $m_\mu(\beta)=0$.)



\end{document}